\documentclass[a4paper,10pt]{amsart}
\usepackage{hyperref}
\hypersetup{
    colorlinks,
    citecolor=black,
    filecolor=black,
    linkcolor=black,
    urlcolor=black
}
\usepackage[utf8]{inputenc}
\usepackage{amssymb,amsthm,amsrefs,amsmath}

\DeclareMathOperator{\dist}{dist}
\DeclareMathOperator{\diam}{diam}

\newcommand{\Z}{\mathbb{Z}}

\newcommand{\K}{\mathcal{K}}

\renewcommand{\epsilon}{\varepsilon}

\newtheorem{thm}{Theorem}[section]
\newtheorem{prop}[thm]{Proposition}
\newtheorem{lemma}[thm]{Lemma}
\newtheorem{rmk}[thm]{Remark}
\newtheorem{coro}[thm]{Corollary}

\theoremstyle{definition}
\newtheorem{df}{Definition}

\title{Finite Sets with Fake Observable Cardinality}

\author{Alfonso Artigue}
\email{artigue@unorte.edu.uy}
\address{Departamento de Matem\'atica y Estad\'\i stica del Litoral, Universidad de la Rep\'ublica, Gral. Rivera 1350, Salto, Uruguay.}
\date{\today}

\begin{document}
\begin{abstract}

Let $X$ be a compact metric space and let $|A|$ denote the cardinality of a set $A$. 
We prove that if $f\colon X\to X$ is a homeomorphism and $|X|=\infty$ then for all 
$\delta>0$ there is $A\subset X$ such that $|A|=4$ and for all $k\in\Z$ there are $x,y\in f^k(A)$, $x\neq y$, such that 
$\dist(x,y)<\delta$. 
An observer that can only distinguish two points if their distance is grater than $\delta$,
for sure will say that $A$ has at most 3 points even knowing every iterate of $A$ and that $f$ is a homeomorphism. 
We show that for hyper-expansive homeomorphisms the same $\delta$-observer will not fail about the cardinality 
of $A$ if we start with $|A|=3$ instead of $4$.
Generalizations of this problem are considered via what we call $(m,n)$-expansiveness.
\end{abstract}
\maketitle

\section*{Introduction}

Since 1950, when Utz \cite{Utz} initiated the study of expansive homeomorphism, 
several variations of the definition appeared in the literature. 
Let us recall that a homeomorphism $f\colon X\to X$ 
of a compact metric space $(X,\dist)$ is \emph{expansive} 
if there is an \emph{expansive constant} $\delta>0$ such that if 
$x\neq y$ then $\dist(f^k(x),f^k(y))>\delta$ for some $k\in\Z$. 
Some variations of this definition are weaker, 
as for example \emph{continuum-wise expansiveness} \cite{Kato93} and 
$N$-\emph{expansiveness} \cite{Morales} (see also
\cites{Morales2011,Reddy70,Bowen72}).
A branch of research in topological dynamics 
investigates the possibility of extending 
known results for expansive homeomorphisms 
to these versions. See for example \cites{Sakai97, PV2008, MoSi, Jana, APV}. 

Other related definitions are stronger than expansiveness 
as for example \emph{positive expansiveness} \cite{Schw} and \emph{hyper-expansiveness} \cite{Artigue}. 
Both definitions are so strong that their examples are almost trivial.
It is known \cite{Schw} that if a compact metric space admits a positive expansive homeomorphism then the space has only a finite number of points. 
Recall that $f \colon X\to X$ is \emph{positive expansive} if there is $\delta>0$ such that 
if $x\neq y$ then $\dist(f^k(x),f^k(y))>\delta$ for some $k\geq 0$. 
Therefore, we have that if the compact metric space $X$ is not a finite set, 
then for every homeomorphism $f\colon X\to X$ 
and for all $\delta>0$ there are 
$x\neq y$ such that $\dist(f^k(x),f^k(y))<\delta$ for all $k\geq 0$. 
This is a very general result about the dynamics of 
homeomorphisms of compact metric spaces.

Another example of this phenomenon is given in \cite{Artigue}, where it is proved that 
no uncountable compact metric space admits a hyper-expansive homeomorphism (see Definition \ref{defHyperExp}).
Therefore, if $X$ is an uncountable compact metric space, 
as for example a compact manifold, then for every 
homeomorphism $f\colon X\to X$ and for all $\delta>0$ there are two compact subsets $A,B\subset X$, $A\neq B$, 
such that $\dist_H(f^k(A),f^k(B))<\delta$ for all $k\in\Z$. 
The distance $\dist_H$ is called \emph{Hausdorff metric} and its definition is recalled in 
equation (\ref{distHaus}) below.

According to Lewowicz \cite{Lew} we can explain the meaning of expansiveness as follows. 
Let us say that a $\delta$-\emph{observer} 
is someone that cannot distinguish two points if their distance is smaller than $\delta$. 
If $\dist(x,y)<\delta$ a $\delta$-observer will not be able to say that the set $A=\{x,y\}$ has two points. 
But if the homeomorphism is expansive, with expansive constant greater than $\delta$, and if the $\delta$-observer 
knows all of the iterates $f^k(A)$ with $k\in \Z$, then he will find that $A$ contains two different points,
because if $\dist(f^k(x),f^k(y))>\delta$ then he will see two points in $f^k(A)$.
Let us be more precise.

\begin{df}
For $\delta\geq 0$, a set $A\subset X$ is $\delta$-\emph{separated} if for all $x\neq y$, $x,y\in A$, it holds that $\dist(x,y)>\delta$. 
The $\delta$-\emph{cardinality} of a set $A$ is 
$$|A|_\delta=\sup\{|B|: B\subset A\hbox{ and } B \hbox{ is }\delta\hbox{-separated}\},$$
where $|B|$ denotes the cardinality of the set $B$.
\end{df}

Notice that the $\delta$-cardinality is always finite because $X$ is compact. 
The $\delta$-cardinality of a set represents the maximum number of different points that a $\delta$-observer can identify in the set.

In this paper we introduce a series of definitions, some weaker and other stronger than expansiveness,
extending the notion of  $N$-expansiveness of \cite{Morales}. 
Let us recall that given $N\geq 1$, 
a homeomorphism is $N$-\emph{expansive} if there is $\delta>0$ such that 
if $\diam(f^k(A))<\delta$ for all $k\in\Z$ then $|A|\leq N$. 
In terms of our $\delta$-observer we can say that $f$ is $N$-expansive 
if there is $\delta>0$
such that if $|A|=N+1$, a $\delta$-observer will be able to say that $A$ has at least two points
given that he knows all 
of the iterates $f^k(A)$ for $k\in\Z$, i.e., $|f^k(A)|_\delta>1$ for some $k\in\Z$. 
Let us introduce our main definition.

\begin{df}
Given integer numbers $m>n\geq 1$ we say that $f\colon X\to X$ is $(m,n)$-\emph{expansive} if 
there is $\delta>0$ such that if $|A|=m$ then there is $k\in \Z$ such that $|f^k(A)|_\delta> n$.
\end{df}

The first problem under study is the classification of these definitions.
We prove that $(m,n)$-expansiveness implies $N$-expansiveness if $m\leq (N+1)n$. 
In particular, if $m\leq 2n$ then $(m,n)$-expansiveness implies expansiveness.
These results are stated in Corollary \ref{mnImplicaExp}. 
It is known that even on surfaces, $N$-expansiveness does not imply expansiveness for $N\geq 2$, see \cite{APV}. 
Here we show that $(m,n)$-expansiveness does not imply expansiveness if $n\geq 2$. 
For example, Anosov diffeomorphisms are known to be expansive and a consequence of Theorem \ref{teoShadowing} 
is that Anosov diffeomorphisms are not $(m,n)$-expansive for all $n\geq 2$.

It is a fundamental problem in dynamical systems to determine which spaces admit expansive homeomorphisms (or Anosov diffeomorphisms). 
In this paper we prove that no Peano continuum admits a $(m,n)$-expansive homeomorphism 
if $2m\geq 3n$, see Theorem \ref{Peano1}. 
We also show that if $X$ admits a $(n+1,n)$-expansive homeomorphism with $n\geq 3$ then $X$ is a finite set. 
Examples of $(3,2)$-expansive homeomorphisms are given on countable spaces (hyper-expansive homeomorphisms), see 
Theorem \ref{hyperexp}.

The article is organized as follows.
In Section \ref{SecSFS} we prove basic properties of $(m,n)$-expansive homeomorphisms. 
In Section \ref{Sec4P} we prove the first statement of the abstract, i.e., no infinite compact metric space admits 
a $(4,3)$-expansive homeomorphism. 
In {Section} \ref{SecPeano} we show that no Peano continuum admits 
a $(m,n)$-expansive homeomorphism if $2m\geq 3n$. 
In {Section} \ref{SecHyper} we show that hyper-expansive homeomorphisms are $(3,2)$-expansive. 
Such homeomorphisms are defined on compact metric spaces with a countable number of points.
In {Section} \ref{SecGenCase} we prove that a homeomorphism with the shadowing property 
and with two points $x,y$ satisfying 
$$0=\liminf_{k\to\infty}\dist(f^k(x),f^k(y))<\limsup_{k\to\infty}\dist(f^k(x),f^k(y))$$
cannot be   
$(m,2)$-expansive if $m>2$. 


\section{Separating Finite Sets}
\label{SecSFS}

Let $(X,\dist)$ be a compact metric space and 
consider a homeomorphism $f\colon X\to X$. 
Let us recall that
for integer numbers $m>n\geq 1$ 
a homeomorphism $f$ is $(m,n)$-\emph{expansive} if there is 
 $\delta>0$ such that if 
 $|A|=m$ then there is $k\in\Z$ 
 such that $|f^k(A)|_\delta> n$. 
In this case we say that $\delta$ is a $(m,n)$-\emph{expansive constant}. 
The idea of $(m,n)$-expansiveness is
that our $\delta$-observer will find more than $n$ points in every set of $m$ points if he knows all of its iterates.

\begin{rmk}
 From the definitions it follows that a homeomorphisms is $(N+1,1)$-expansive if and only if it is $N$-expansive in the sense of \cite{Morales}. 
 In particular, $(2,1)$-expansiveness is equivalent with expansiveness.
\end{rmk}

\begin{rmk}
Notice that if $X$ is a finite set then every homeomorphism of $X$ is $(m,n)$-expansive. 
\end{rmk}

\begin{prop}
\label{propRelBas}
 If $n'\leq n$ and $m-n\leq m'-n'$ then $(m,n)$-expansive implies $(m',n')$-expansive with the same expansive constant.
\end{prop}

\begin{proof}
 The case $|X|<\infty$ is trivial, so, let us assume that $|X|=\infty$. 
Consider $\delta>0$ as a $(m,n)$-expansive constant. 
Given a set $A$ with $|A|=m'$ we will show that there is $k\in\Z$ such that 
$|f^k(A)|_\delta>n'$, i.e., the same expansive constant works. 
We divide the proof in two cases. 

First assume that $m'\geq m$. 
Let $B\subset A$ with $|B|=m$. 
Since $f$ is $(m,n)$-expansive, there is $k\in\Z$ such that 
$|f^k(B)|_\delta>n$. 
Therefore $|f^k(A)|_\delta> n\geq n'$, proving that $f$ is $(m',n')$-expansive. 

Now suppose that $m'<m$. 
Given that $|A|=m'$ and $|X|=\infty$ there is $C\subset X$ such that 
$A\cap C=\emptyset$ and $|A\cup C|=m$. 
By $(m,n)$-expansiveness, there is $k\in\Z$ such that 
$|f^k(A\cup C)|_\delta>n$. 
Then, there is a $\delta$-separated set $D\subset f^k(A\cup C)$ with $|D|>n$.
Notice that 
$$|f^k(A)\cap D|= |D\setminus f^k(C)|\geq |D|-|f^k(C)|>n-(m-m')$$
and since $n-(m-m')\geq n'$ by hypothesis, we have that 
$f^k(A)\cap D$ is a $\delta$-separated subset of $f^K(A)$ with more than $n'$ points, 
that is $|f^k(A)|_\delta>n'$. 
This proves the $(m'n')$-expansiveness of $f$ in this case too.
\end{proof}

As a consequence of Proposition \ref{propRelBas} we have that 
\begin{enumerate}
 \item $(m,n)$-expansive implies $(m+1,n)$-expansive and 
 \item $(m,n)$-expansive implies $(m-1,n-1)$-expansive.
\end{enumerate}
In Table \ref{tabla} below we can easily see all  these implications. 
The following proposition allows us to draw more arrows in this table, for example: $(4,2)\Rightarrow (2,1)$.

\begin{table}[h]
\label{tabla}
\caption{Basic hierarchy of $(m,n)$-expansiveness. 
Each pair $(m,n)$ in the table stands for ``$(m,n)$-expansive``.
In the first position, (2,1), we have expansiveness. 
The first line, of the form $(N+1,1)$, we have $N$-expansive homeomorphisms. 
}
\[
\begin{array}{ccccccc}
(2,1)    & \Rightarrow & (3,1)    & \Rightarrow & (4,1)    & \Rightarrow & \dots  \\
\Uparrow &             & \Uparrow &             & \Uparrow &             & \\
(3,2)    & \Rightarrow & (4,2)    & \Rightarrow & (5,2)    & \Rightarrow & \dots  \\
\Uparrow &             & \Uparrow &             & \Uparrow &             & \\
(4,3)    & \Rightarrow & (5,3)    & \Rightarrow & (6,3)    & \Rightarrow & \dots  \\
\Uparrow &             & \Uparrow &             & \Uparrow &             & \\
\dots    &             & \dots    &             & \dots    &             &
\end{array}
\]
\label{tableExp}
\end{table}

\begin{prop}
\label{prop(an,n)}
 If $a,n\geq 2$ and $f\colon X\to X$ is an $(an,n)$-expansive homeomorphism
 then $f$ is $(a,1)$-expansive.
\end{prop}

In order to prove it, let us introduce two previous results. 

\begin{lemma}
\label{lemma1(an,n)}
 If $A,B\subset X$ are finite sets and $\delta>0$ satisfies $|A|=|A|_\delta$ and $|B|_\delta=1$ then 
 for all $\epsilon>0$ it holds that 
 \[
  |A\cup B|_{\delta+\epsilon}\leq |A|_\epsilon+|B|_\delta-|A\cap B|.
 \]
\end{lemma}

\begin{proof}
 If $A\cap B=\emptyset$ then the proof is easy because 
 \[
  |A\cup B|_{\delta+\epsilon}\leq |A|_{\delta+\epsilon}+|B|_{\delta+\epsilon}\leq |A|_\epsilon + |B|_\delta.
 \]
Assume now that $A\cap B\neq\emptyset$. 
Since $|A|=|A|_\delta$ we have that $A$ is $\delta$-separated. 
Therefore $|A\cap B|=1$ because $|B|_\delta=1$. 
Assume that $A\cap B=\{y\}$. 
Let us prove that $|A\cup B|_{\delta+\epsilon}\leq|A|_\epsilon$ and notice that 
it is sufficient to conclude the proof of the lemma. 

Let $C\subset A\cup B$ be a $(\delta+\epsilon)$-separated set such that 
$|C|=|A\cup B|_{\delta+\epsilon}$. 
If $C\subset A$ then $$|A\cup B|_{\delta+\epsilon}=|A|_{\delta+\epsilon}\leq |A|_\epsilon.$$
Therefore, let us assume that there is $x\in C\setminus A$. 
Define the set 
\[
 D=(C\cup\{y\})\setminus \{x\}.
\]
Notice that $|C|=|D|$ and $D\subset A$. 

We will show that $D$ is $\epsilon$-separated. 
Take $p,q\in D$ and arguing by contradiction assume that 
$p\neq q$ and $\dist(p,q)\leq\epsilon$. 
If $p,q\in C$ there is nothing to prove because $C$ is $(\delta+\epsilon)$-separated. 
Assume now that $p=y$. 
We have that $\dist(x,p)\leq\delta$ because $x,p\in B$ and $|B|_\delta=1$. 
Thus 
\[
 \dist(x,q)\leq \dist(x,p)+\dist(p,q) \leq \epsilon +\delta.
\]
But this is a contradiction because $x,q\in C$ 
and $C$ is $(\epsilon + \delta)$-separated.
\end{proof}

\begin{lemma}
\label{lemma2propann}
If $f$ is $(m+l,n+1)$-expansive then $f$ is $(m,n)$-expansive or $(l,1)$-expansive.
\end{lemma}

\begin{proof}
Let us argue by contradiction and take an $(m+l,n+1)$-expansive constant $\alpha>0$. 
Since $f$ is not $(m,n)$-expansive for $\epsilon\in(0,\alpha)$ there 
is a set $A\subset X$ such that 
$|A|=m$ and 
$|f^k(A)|_\epsilon \leq n$ for all $k\in\Z$. 
Take $\delta>0$ such that $|A|=|A|_\delta$ and $\delta+\epsilon<\alpha$. 

Since $f$ is not $(l,1)$-expansive there is $B$ such that 
$|B|=l$ and $|f^k(B)|_\delta=1$ for all $k\in\Z$. 
By Lemma \ref{lemma1(an,n)} we have that 
\[
 |f^k(A\cup B)|_{\delta+\epsilon}\leq |f^k(A)|_\epsilon + |f^k(B)|_\delta-|A\cap B|
 \leq n+1-|A\cap B|,
\]
for all $k\in\Z$.
Also, we know that $|A\cup B|=m+l-|A\cap B|$. 
If we denote $r=|A\cap B|$ then $f$ is not 
$(m+l-r,n+1-r)$-expansive. 
And by Proposition \ref{propRelBas} we conclude that 
$f$ is not $(m+l,n+1)$-expansive.
This contradiction proves the lemma.
\end{proof}

\begin{proof}[Proof of Proposition \ref{prop(an,n)}]
Assume by contradiction that $f$ is not $(a,1)$-expansive. 
Since $f$ is $(an,n)$-expansive, by Lemma \ref{lemma2propann} 
we have that $f$ has to be $(a(n-1),n-1)$-expansive. 
Arguing inductively we can prove that 
$f$ is $(a(n-j),n-j)$-expansive, for $j=1,2,\dots, n-1$. 
In particular, $f$ is $(a,1)$-expansive, which is a contradiction 
that proves the proposition.
\end{proof}

\begin{coro}
\label{mnImplicaExp}
If $m\leq an$ and $f$ is $(m,n)$-expansive then $f$ is $(a,1)$-expansive (i.e. $(a-1)$-expansive in the sense of \cite{Morales}).
In particular, if $m\leq 2n$ and $f$ is $(m,n)$-expansive then $f$ is expansive.
\end{coro}

\begin{proof}
 By Proposition \ref{propRelBas} we have that $f$ is $(an,n)$-expansive. 
 Therefore, by Proposition \ref{prop(an,n)} we have that $f$ is $(a,1)$-expansive. 
\end{proof}

\section{Separating 4 points}
\label{Sec4P}

In this section we prove that $(n+1,n)$-expansiveness with $n\geq 3$ implies that $X$ is finite.

\begin{thm}
\label{teoF3}
 If $X$ is a compact metric space admitting a $(4,3)$-expansive homeomorphism then $X$ is a finite set.
\end{thm}

\begin{proof}
 By contradiction assume that $f$ is a $(4,3)$-expansive homeomorphism of $X$ with $|X|=\infty$ and take an expansive constant $\delta>0$.
 We know that $f$ cannot be positive expansive (see \cites{Lew,CK} for a proof).
 Therefore there are $x_1,x_2$ such that $x_1\neq x_2$ and
 \begin{equation}\label{x12}
 \dist(f^k(x_1),f^k(x_2))<\delta 
 \end{equation}
 for all $k\geq 0$. 
 Analogously, $f^{-1}$ is not positive expansive, and we can take
 $y_1,y_2$ such that $y_1\neq y_2$ and
 \begin{equation}\label{y12}
 \dist(f^k(y_1),f^k(y_2))<\delta 
 \end{equation}
 for all $k\leq 0$. 
 Consider the set $A=\{x_1,x_2,y_1,y_2\}$. 
 We have that $2\leq|A|\leq 4$ (we do not know if the 4 points are different).
 By inequalities (\ref{x12}) and (\ref{y12}) we have that $|f^k(A)|_\delta< |A|$ for all $k\in\Z$.
 If $n=|A|$ then we have that $f$ is not $(n,n-1)$-expansive. 
 In any case, $n=2,3$ or $4$, 
 by Proposition \ref{propRelBas} (see Table \ref{tableExp}) we conclude that $f$ is not $(4,3)$-expansive.
 This contradiction finishes the proof.
\end{proof}

\begin{rmk}
\label{rmkFiniteX}
 If $X$ is a compact metric space admitting a $(n+1,n)$-expansive homeomorphism with $n\geq 3$ 
 then $X$ is a finite set. It follows by Theorem \ref{teoF3} and Proposition \ref{propRelBas}.
\end{rmk}

\begin{coro}
 If $f\colon X\to X$ is a homeomorphism of a compact metric space and $|X|=\infty$ then for all 
 $\delta>0$ and $m\geq 4$ there is $A\subset X$ with $|A|=m$ such that 
 $|f^k(A)|_\delta<|A|$
 for all $k\in\Z$. 
\end{coro}

\begin{proof}
 It is just a restatement of Remark \ref{rmkFiniteX}.
\end{proof}

\section{On Peano continua}
\label{SecPeano}

In this section we study $(m,n)$-expansiveness on Peano continua.
Let us start recalling 
that a \emph{continuum} is a compact connected metric space and a 
\emph{Peano continuum} is a locally connected continuum. 
A singleton space ($|X|=1$) is a \emph{trivial} Peano continuum.
For $x\in X$ and $\delta>0$ define the \emph{stable} and \emph{unstable} set of $x$ as
\[
 \begin{array}{l}
 W^s_\delta(x)=\{y\in X:\dist(f^k(x),f^k(y))\leq\delta\,\forall \,k\geq 0\}, \\
 W^u_\delta(x)=\{y\in X:\dist(f^k(x),f^k(y))\leq\delta\,\forall \,k\leq 0\}.
 \end{array}
\]

\begin{rmk}
Notice that $(m,n)$-expansiveness implies continuum-wise expansiveness for all $m>n\geq 1$. 
Recall that $f$ is \emph{continuum-wise expansive} if 
there is $\delta>0$ such that if $\diam(f^k(A))<\delta$ for all $k\in\Z$ and some continuum $A\subset X$, 
then $|A|=1$. 
\end{rmk}

\begin{thm}
\label{Peano1}
 If $X$ is a non-trivial Peano continuum then no homeomorphism of $X$ is $(m,n)$-expansive if $2m\geq 3n$. 
\end{thm}

\begin{proof}
 Let $\delta$ be a positive real number and assume that $f$ is $(m,n)$-expansive.
 As we remarked above, $f$ is a continuum-wise expansive homeomorphism.
 It is known (see \cites{Kato93,Jana}) that for such homeomorphisms on a Peano continuum, 
 every point has 
 non-trivial stable and unstable sets. 
 Take $n$ different points $x_1,\dots,x_n\in X$ and let $\delta'\in(0,\delta)$ be such that 
 $\dist(x_i,x_j)>2\delta'$ if $i\neq j$. 
 For each $i=1,\dots,n$, we can take $y_i\in W^s_{\delta'}(x_i)$ and 
 $z_i\in W^u_{\delta'}(x_i)$ with $x_i\neq y_i$ and $x_i\neq z_i$. 
 Consider the set $A=\{x_1,y_1,z_1,\dots, x_n,y_n,z_n\}$. 
 Since $\dist(x_i,x_j)>2\delta'$ if $i\neq j$, and $y_i,z_i\in B_{\delta'}(x_i)$ we have that 
 $|A|=3n$. 
 If $A_i$ denotes the set $\{x_i,y_i,z_i\}$ we have that 
 $|f^k(A_i)|_{\delta'}\leq2$ for all $k\in\Z$. 
 That is because if $k\geq 0$ then $\dist(f^k(x_i),f^k(y_i))\leq\delta'$ and if 
 $k\leq 0$ then $\dist(f^k(x_i),f^k(z_i))\leq\delta'$.
 Therefore $|f^k(A)|_{\delta'}\leq 2n$, and then
$|f^k(A)|_\delta\leq 2n$. 
 Since $\delta>0$ and $n\geq 1$ are arbitrary, we have that $f$ is not $(3n,2n)$ expansive for all $n\geq 1$. 
 Finally, by Proposition \ref{propRelBas}, we have that $f$ is not $(m,n)$-expansive if $2m\geq 3n$.
\end{proof}

\begin{coro}
 If $f\colon X\to X$ is a homeomorphism and $X$ is a non-trivial Peano continuum then for all $\delta>0$ 
 there is $A\subset X$ such that $|A|=3$ and $|f^k(A)|_\delta\leq 2$
 for all $k\in\Z$.
\end{coro}

\begin{proof}
 By Theorem \ref{Peano1} we know that $f$ is not $(3,2)$-expansive. 
 Therefore, the proof follows by definition.
\end{proof}

\section{Hyper-expansive homeomorphisms}
\label{SecHyper}
Denote by $\K(X)$ the set of compact subsets of $X$.
This space is usually called as the \emph{hyper-space} of $X$. 
We recommend the reader to see \cite{Nadler} for more on the subject of hyper-spaces and 
the proofs of the results that we will cite below.
In the set $\K(X)$ we consider the Hausdorff distance 
$\dist_H$ making $(\K(X),\dist_H)$ a compact metric space. 
Recall that 
\begin{equation}\label{distHaus}
 \dist_H(A,B)=\inf\{\epsilon>0:A\subset B_\epsilon(B)\hbox{ and } B\subset B_\epsilon(A)\},
\end{equation}
where $B_\epsilon(C)=\cup_{x\in C} B_\epsilon(x)$ and $B_\epsilon(x)$ is the usual ball of radius $\epsilon$ 
centered at $x$.
As usual, we let $f$ to act on $\K(X)$ as $f(A)=\{f(a):a\in A\}$.
\begin{df}
\label{defHyperExp}
We say that $f$ is \emph{hyper-expansive} if $f\colon\K(X)\to \K(X)$ is expansive, i.e., there is $\delta>0$ such that 
given two compact sets $A,B\subset X$, $A\neq B$, there is $k\in\Z$ such that $\dist_H(f^k(A),f^k(B))>\delta$ 
where $\dist_H$ is the Hausdorff distance. 
\end{df}
In \cite{Artigue} it is shown that $f\colon X\to X$ is hyper-expansive if and only if 
$f$ has a finite number of orbits (i.e., there is a finite set $A\subset X$ such that $X=\cup_{k\in\Z}f^k(A)$) 
and the non-wandering set is a finite union of periodic points which are attractors or repellers. 
Recall that a point $x$ is in the \emph{non-wandering set} if for every neighborhood $U$ of $x$ there is $k>0$ 
such that $f^k(U)\cap U\neq\emptyset$. 
A point $x$ is \emph{periodic} if for some $k\geq 0$ it holds that $f^k(x)=x$. 
The orbit $\gamma=\{x,f(x),\dots,f^{k-1}(x)\}$ is a \emph{periodic orbit} if $x$ is a periodic point. 
A periodic orbit $\gamma$ is an \emph{attractor} (\emph{repeller}) if there is a compact neighborhood $U$ of $\gamma$ such that 
$f^k(U)\to\gamma$ in the Hausdorff distance as $k\to\infty$ (resp. $k\to-\infty$).

\begin{thm}\label{hyperexp}
 If $f\colon X\to X$ is a hyper-expansive homeomorphism and $|X|=\infty$ then 
 $f$ is $(m,n)$-expansive for some $m>n\geq 1$ if and only if $m\leq 3$.
\end{thm}

\begin{proof}
Let us start with the direct part of the theorem. 
Let $P_a$ be the set of periodic attractors, $P_r$ the set of periodic repellers and 
take $x_1,\dots, x_j$ one point in each wandering orbit (recall that, as we said above, 
hyper-expansiveness implies that $f$ has just a finite number of orbits).
Define $Q=\{x_1,\dots,x_j\}$.
Take $\delta>0$ such that 
\begin{enumerate}
 \item if $p,q\in P_a\cup P_r$ and $p\neq q$ then $\dist(p,q)>\delta$,
 \item if $x_i\in Q$ then $B_\delta(x_i)=\{x_i\}$ (recall that wandering points are isolated by \cite{Artigue}),
 \item if $p\in P_a$, $x_i\in Q$ and $k\leq 0$ then $\dist(p,f^k(x_i))>\delta$, 
 \item if $q\in P_r$, $x_i\in Q$ and $k\geq 0$ then $\dist(p,f^k(x_i))>\delta$ and 
 \item if $x,y\in Q$ and $k>0>l$ then $\dist(f^k(x),f^l(y))>\delta$.
\end{enumerate}
Let us prove that such $\delta$ is a $(3,2)$-expansive constant.
Take $a,b,c\in X$ with $|\{a,b,c\}|=3$. The proof is divided by cases: 
\begin{itemize}
 \item If $a,b,c\in P=P_a\cup P_r$ then item 1 above concludes the proof. 
 \item If $a,b\in P$ and $c\notin P$ then there is $k\in\Z$ such that $f^k(c)\in Q$. 
 In this case items 1 and 2 conclude the proof. 
 \item Assume now that $a\in P$ and $b,c\notin P$. Without loss of generality 
 let us suppose that $a$ is a repeller. Let $k_b,k_c\in\Z$ be such that 
 $f^{k_b}(b),f^{k_c}(c)\in Q$. 
 Define $k=\min\{k_b,k_c\}$. 
 In this way: $\dist(f^k(a),f^k(b)),\dist(f^k(a),f^k(c))\geq\delta$ by item 4 and 
 $\dist(f^k(b),f^k(c))\geq\delta$ by item 2. 
 \item If $a,b,c\notin P$ then
 consider $k_a,k_b,k_c\in\Z$ such that $f^{k_a}(a),f^{k_b}(b),f^{k_c}(c)\in Q$. 
 Assume, without loss, that $k_a\leq k_b\leq k_c$. 
 Take $k=k_b$. 
 In this way, items 2 and 5 finishes the direct part of the proof.
\end{itemize}

To prove the converse, 
we will show that $f$ is not $(m,3)$-expansive for all $m>3$.
Take $\delta>0$.
Notice that since $X=\infty$ there is at least one wandering point $x$. 
Without loss of generality assume that $\lim_{k\to\infty} f^k(x)=p_a$ an attractor fixed point 
and $\lim_{k\to-\infty} f^k(x)=p_r$ a repeller fixed point. 
Take $k_1,k_2\in\Z$ such that $\dist(f^k(x),p_r)<\delta$ for all $k\leq k_1$ and 
$\dist(f^k(x),p_a)<\delta$ for all $k\geq k_2$. 
Let $l=k_2-k_1$ and define $x_1=f^{-k_1}(x)$, and $x_{i+1}=f^l(x_i)$ for all $i\geq 1$. 
Consider the set $A=\{x_1,\dots,x_m\}$. 
By construction we have that $|A|=m$ and $|f^k(A)|_\delta\leq 3$ for all $k\in\Z$. 
Thus, proving that $f$ is not $(m,3)$-expansive if $m>3$.
\end{proof}

\begin{rmk}
In light of the previous proof one may wonder if a \emph{smart} $\delta$-observer will not be able to 
say that $A$ has more than 3 points. 
We mean, we are assuming that a $\delta$-observer will say that 
$A$ has $n'$ points with 
\[
 n'=\max_{k\in\Z} |f^k(A)|_\delta.
\]
According to the dynamic of the set $A$ in the previous proof, we guess that with more reasoning a smarter $\delta$-observer 
will find that $A$ has more than 3 points.
\end{rmk}

Theorem \ref{hyperexp} gives us examples of $(3,2)$-expansive homeomorphisms on infinite countable 
compact metric spaces. 
A natural question is: does $(3,2)$-expansiveness implies hyper-expansiveness?
I do not know the answer, but let us remark some facts that may be of interest. 
If $f$ is $(3,2)$-expansive then:
\begin{itemize}
 \item For all $x\in X$ either the stable or the unstable set must be trivial. 
 It follows by the arguments of the proof of Theorem \ref{Peano1}. 
 \item If $x,y$ are bi-asymptotic, i.e., $\dist(f^k(x),f^k(y))\to 0$ as $k\to\pm\infty$ then 
 they are isolated points of the space. Suppose that $x$ were an accumulation point. 
 Given $\delta>0$ take $k_0$ such that if $|k|>k_0$ then $\dist(f^k(x),f^k(y))<\delta$. 
 Take a point $z$ close to $x$ such that $\dist(f^k(x),f^k(z))<\delta$ if $|k|\leq k_0$ (we are just using the continuity of $f$). 
 Then $x,y,z$ contradicts $(3,2)$-expansiveness.
\end{itemize}

\begin{prop}
 There are $(4,2)$-expansive homeomorphisms that are not $(3,2)$-expansive.
\end{prop}
\begin{proof}
 Let us prove it giving an example. 
 Consider a countable compact metric space $X$ and a homeomorphism $f\colon X\to X$ 
 with the following properties:
 \begin{enumerate}
  \item $f$ has 5 orbits, 
  \item $a,b,c\in X$ are fixed points of $f$,
  \item there is $x\in X$ such that $\lim_{k\to-\infty}f^k(x)=a$ and 
  $\lim_{k\to+\infty}f^k(x)=b$,
  \item there is $y\in X$ such that $\lim_{k\to-\infty}f^k(y)=b$ and 
  $\lim_{k\to+\infty}f^k(y)=c$.
 \end{enumerate}
In order to see that $f$ is not $(3,2)$-expansive consider $\epsilon>0$. 
Take $k_0\in\Z$ such that for all $k\geq k_0$ it holds that 
$\dist(f^k(x),b)<\epsilon$ and $\dist(f^{-k}(y),b)<\epsilon$. 
Define $u=f^{k_0}(x)$ and $v=f^{-k_0}(y)$. 
In this way $\|\{f^k(u),b,f^k(v)\}\|_\epsilon\leq 2$ for all $k\in\Z$.
This proves that $f$ is not $(3,2)$-expansive. 

 
Let us now indicate how to prove that $f$ is $(4,2)$-expansive. 
Consider $\epsilon>0$ such that if $i\geq 0$ and $j\in\Z$ then 
$\dist(f^{-i}(x),f^j(y))>\epsilon$ and 
$\dist(f^j(x),f^i(y))>\epsilon$. 
Now, a similar argument as in the proof of Theorem \ref{hyperexp}, shows that $f$ is $(4,2)$-expansive.
\end{proof}

\section{On the general case}
\label{SecGenCase}
In this section we prove that an important class of homeomorphisms are not $(m,n)$-expansive 
for all $m>n\geq 2$.
In order to state this result let us recall that a $\delta$-\emph{pseudo orbit} 
is a sequence $\{x_k\}_{k\in\Z}$ such that 
$\dist(f(x_k),x_{k+1})\leq\delta$ for all $k\in\Z$. 
We say that a homeomorphism has the \emph{shadowing property} 
if for all $\epsilon>0$ there is $\delta>0$ such that 
if $\{x_k\}_{k\in\Z}$ is a $\delta$-pseudo orbit 
then there is $x$ such that 
$\dist(f^k(x),x_k)<\epsilon$ for all $k\in\Z$.
In this case we say that $x$ $\epsilon$-\emph{shadows} the $\delta$-pseudo orbit.

\begin{thm}
\label{teoShadowing}
 Let $f\colon X\to X$ be a homeomorphism of a compact metric space $X$. 
 If $f$ has the shadowing property and
 there are $x,y\in X$ such that 
 $$0=\liminf_{k\to\infty}\dist(f^k(x),f^k(y))<\limsup_{k\to\infty}\dist(f^k(x),f^k(y))$$
 then $f$ is not $(m,n)$-expansive if $m>n\geq 2$.
\end{thm}

\begin{proof}
 By Proposition \ref{propRelBas} it is enough to prove that $f$ cannot be $(m,2)$-expansive if $m>2$.
 Consider $\epsilon>0$. 
 We will define a set $A$ with $|A|=\infty$ such that for all $k\in\Z$, 
 $f^k(A)\subset B_\epsilon(f^k(x))\cup B_\epsilon(f^k(y))$, proving that $f$ is not $(m,2)$-expansive for all $m>2$. 
  
 Consider two increasing sequences $a_l,b_l \in\Z$, $\rho\in(0,\epsilon)$ and $\delta>0$ such that 
 \[
  \begin{array}{l}
 a_1<b_1<a_2<b_2<a_3<b_3<\dots,\\ 
 \dist(f^{a_l}(x),f^{a_l}(y))<\delta,\\
 \dist(f^{b_l}(x),f^{b_l}(y))>\rho
  \end{array}
 \]
 for all $l\geq 1$ and assume that every $\delta$-pseudo orbit can be $(\rho/2)$-shadowed.  
 For each $l\geq 1$ define the $\delta$-pseudo orbit $z^l_k$ as 
 \[
 z^l_k=\left\{
  \begin{array}{l}
    f^k(x)\hbox{ if }k<a_l,\\
    f^k(y)\hbox{ if }k\geq a_l. 
  \end{array}
  \right.
 \]
 Then, for every $l\geq 1$ there is a point $w^l$ whose orbit $(\rho/2)$-shadows the $\delta$-pseudo orbit $\{z^l_k\}_{k\in\Z}$. 
 Let us now prove that if $1\leq l<s$ then $w^l\neq w^{s}$. We have that $a_l<b_l<a_{s}$.
 Therefore $z^l_{b_l}=f^{b_l}(y)$ and $z^{s}_{b_l}=f^{b_l}(x)$. 
 Since $w^l$ and $w^{s}$ $(\rho/2)$-shadows the pseudo orbits $z^l$ and $z^{s}$ respectively, we have that 
 $$\dist(f^{b_l}(w^l),f^{b_l}(y)),\dist(f^{b_l}(w^{s}),f^{b_l}(x))<\rho/2.$$ 
 We conclude that $w^l\neq w^{s}$ because $\dist(f^{b_l}(x),f^{b_l}(y))>\rho$.
 Therefore, if we define $A=\{w^l:l\geq 1\}$ we have that $|A|=\infty$.
 Finally, since $\rho<\epsilon$, we have that 
 $f^k(A)\subset B_\epsilon(f^k(x))\cup B_\epsilon(f^k(y))$ for all $k\in\Z$. 
 Therefore, $|f^k(A)|_\epsilon\leq 2$ for all $k\in\Z$.
 \end{proof}

\begin{rmk}
Examples where Theorem \ref{teoShadowing} can be applied are Anosov diffeomorphisms and symbolic shift maps. 
Also, if $f\colon X\to X$ is a homeomorphism with an invariant set $K\subset X$ such that 
  $f\colon K\to K$ is conjugated to a symbolic shift map then Theorem 
  \ref{teoShadowing} holds because the $(m,n)$-expansiveness of $f$ in $X$ implies the $(m,n)$-expansiveness 
of $f$ restricted to any compact invariant set $K\subset X$ as can be easily checked.   
\end{rmk}

\end{document}